\documentclass[11pt]{article}
\usepackage[top=2cm,bottom=2cm,left=2cm,right=2cm]{geometry}
\usepackage[dvipdfm]{graphicx,color}
\usepackage{amsmath,amsthm,amssymb}
\usepackage{enumerate}

\numberwithin{equation}{section}
\newtheorem{theorem}{Theorem}[section]
\newtheorem{lemma}[theorem]{Lemma}

\def\a{\alpha}

\def\d{\delta}

\def\G{\Gamma}

\def\l{\lambda}

\def\th{\theta}

\def\vp{\varphi}

\def\cut{\setminus}

\DeclareMathOperator\re{{Re}} \DeclareMathOperator\im{{Im}}

\def\sinh{\mbox{sinh}}
\def\cosh{\mbox{cosh}}

\def\arctanh{\mbox{arctanh}}

\def\AND{\qquad\mbox{and}\qquad}

\def\C{\mathbb C}

\def\N{\mathbb N}

\begin{document}
\title{Asymptotics of Orthogonal Polynomials via Recurrence Relations}

\author{X.-S. Wang\thanks{Corresponding author.
Department of Mathematics and Statistics, York University, Toronto, Ontario, Canada}~ and
R. Wong\thanks{Department of Mathematics, City University of Hong Kong, Tat Chee Avenue, Kowloon, Hong Kong}}

\date{}

\maketitle

\begin{abstract}
We use the Legendre polynomials and the Hermite polynomials as two examples to illustrate a simple and systematic technique on deriving asymptotic formulas for orthogonal polynomials via recurrence relations. Another application of this technique is to provide a solution to a problem recently raised by M. E. H. Ismail.
\end{abstract}


\bigskip


\newpage

\section{Introduction}
There are many powerful and systematically developed techniques in asymptotic theory for orthogonal polynomials. For instance, the steepest-descent method for integrals \cite{Wo89},
the WKB (Liouville-Green) approximation for differential equations \cite{Ol97},
the Deift-Zhou's nonlinear steepest-descent method for Riemann-Hilbert problems \cite{DKMVZ99-0,DKMVZ99},
and etc. Here, we intend to develop a simple, and yet systematic approach to derive asymptotic formulas for orthogonal polynomials by using their recurrence relations.
Let $\{\pi_n(x)\}_{n=0}^\infty$ be a system of monic polynomials satisfying the recurrence relation
\begin{eqnarray}\label{rr}
\pi_{n+1}(x)=(x-a_n)\pi_n(n)-b_n\pi_{n-1}(x),\qquad n\ge1,
\end{eqnarray}
and the initial conditions $\pi_0(x)=1$ and $\pi_1(x)=x-a_0$.
Note that for the sake of convenience, we have normalized the polynomials to be monic.
To construct the asymptotic formulas of $\pi_n(x)$, we first set
\begin{eqnarray}\label{wk}
\pi_n(x)=\prod_{k=1}^nw_k(x).
\end{eqnarray}
It is readily seen from (\ref{rr}) that $w_1(x)=x-a_0$ and
\begin{eqnarray}\label{wk-rr}
w_{k+1}(x)=x-a_k-{b_k\over w_k},\qquad k\ge1.
\end{eqnarray}
When $x$ is away from the oscillatory region of the orthogonal polynomials, it is easy to find an asymptotic formula for $w_k(x)$ from (\ref{wk-rr}). Then, as we shall see, the asymptotic behavior of $\pi_n(x)$ for $x$ away from the oscillatory region can be obtained readily.
When $x$ is near the oscillatory region, we use a method similar to that given in \cite{WW02,WL92} to derive asymptotic formulas for general solutions of (\ref{rr}).
The asymptotic formula of $\pi_n(x)$ for $x$ near the oscillatory region is then obtained by doing a matching.
In the subsequent three sections, we will consider the following three cases:
\begin{enumerate}[\mbox{Case} 1:]
  \item $a_n=0$ and $b_n=n^2/(4n^2-1)$. This case is related to the Legendre polynomials.
  \item $a_n=0$ and $b_n=n/2$. This case is related to the Hermite polynomials.
  \item $a_n=n^2$ and $b_n=1/4$. This case was recently brought to our attention by M. E. H. Ismail.
\end{enumerate}
For simplicity, we use the same notations in the following three sections.
Because each section is independent and self-contained, this will not lead to any confusion.

\section{Case 1: the Legendre polynomials}
The Legendre polynomials can be defined as \cite[(1.8.57)]{KS98}
$$P_n(x)={}_2F_1\left(\begin{array}{c}
    -n,n+1\\1
  \end{array}
  \bigg|{1-x\over2}\right).$$
They satisfy the recurrence relation \cite[(1.8.59)]{KS98}
$$(2n+1)xP_n(x)=(n+1)P_{n+1}(x)+nP_{n-1}(x).$$
For convenience, we normalize the Legendre polynomials to be monic. Put
$$\pi_n(x):={2^nn!\over(n+1)_n}P_n(x).$$
The monic Legendre polynomials $\{\pi_n(x)\}_{n=0}^\infty$ satisfy \cite[(1.8.60)]{KS98}
\begin{eqnarray}
&&\pi_{n+1}(x)=x\pi_n(x)-\frac{n^2}{4n^2-1}\pi_{n-1}(x),\qquad n\ge1,\label{L-rr}\\
&&\pi_0(x)=1,\qquad \pi_1(x)=x.\label{L-ic}
\end{eqnarray}
\begin{theorem}
As $n\to\infty$, we have
\begin{eqnarray}\label{L-out}
\pi_n(x)\sim\left(\frac{x+\sqrt{x^2-1}}2\right)^n
\left({x+\sqrt{x^2-1}\over2\sqrt{x^2-1}}\right)^{1/2}
\end{eqnarray}
for $x$ in the complex plane bounded away from $[-1,1]$.
\end{theorem}
\begin{proof}
Set
\begin{eqnarray}\label{L-wk}
\pi_n(x)=\prod_{k=1}^nw_k(x).
\end{eqnarray}
From (\ref{L-rr}), (\ref{L-ic}) and (\ref{L-wk}), it follows that
\begin{eqnarray}
&&w_{k+1}(x)=x-\frac{k^2}{4k^2-1}\frac1{w_k(x)},\qquad k\ge1,\label{L-wk-rr}\\
&&w_1(x)=x.\label{L-wk-ic}
\end{eqnarray}
As $k\to\infty$, we have
$$w_k(x)\sim\frac{x+\sqrt{x^2-1}}2$$
for $x\in\C\cut[-1,1]$.
Here the square root takes its principle value so that $\sqrt{x^2-1}\sim x$ as $x\to\infty$.
Define
\begin{eqnarray}\label{L-w}
w(x):={x+\sqrt{x^2-1}\over2}
\end{eqnarray}
and
\begin{eqnarray}\label{L-uk}
u_k(x):={w_k(x)\over w(x)}.
\end{eqnarray}
It is easily seen from (\ref{L-wk-rr}), (\ref{L-wk-ic}) and (\ref{L-uk}) that
\begin{eqnarray}
&&u_{k+1}(x)=\frac{x}{w(x)}-\frac{k^2}{4k^2-1}\frac1{w(x)^2u_k(x)},\qquad k\ge1,
\label{L-ukx}\\
&&u_1(x)={x\over w(x)}.\label{L-u1x}
\end{eqnarray}
We make a change of variable
\begin{eqnarray}\label{L-t}
t=t(x):=(x-\sqrt{x^2-1})^2.
\end{eqnarray}
It follows from (\ref{L-w}) and (\ref{L-t}) that
\begin{eqnarray}
w(x)^2={1\over4t},\qquad \frac x{w(x)}=1+t.
\end{eqnarray}
Hence, the equations (\ref{L-ukx}) and (\ref{L-u1x}) can be written as
\begin{eqnarray}
&&u_{k+1}(x)=1+t-\frac{4k^2t}{4k^2-1}\frac1{u_k(x)},\qquad k\ge1,\label{L-uk-rr}\\
&&u_1(x)=1+t.\label{L-uk-ic}
\end{eqnarray}
Define $Q_0(t):=1$ and
\begin{eqnarray}\label{L-Q}
Q_n(t):=\prod_{k=1}^nu_k(x),\qquad n\ge1.
\end{eqnarray}
From (\ref{L-uk-rr}), (\ref{L-uk-ic}) and (\ref{L-Q}), we obtain $Q_1(t)=1+t$ and
$$Q_{n+1}(t)=(1+t)Q_n(t)-\frac{4n^2t}{4n^2-1}Q_{n-1}(t).$$
From this recurrence relation, one can construct a generating function from which it is easily deducible that $Q_n(t)$ has the explicit expression
$$Q_n(t)=\sum_{j=0}^n\frac{(1/2)_j(n-j+1)_j}{j!(n-j+1/2)_j}t^j.$$
A simpler verification of this identity is by induction.
Using the Lebesgue dominated convergence theorem,
it can be readily shown that
$$Q_n(t)\to(1-t)^{-1/2}$$
as $n\to\infty$.
Note that by (\ref{L-wk}), (\ref{L-uk}) and (\ref{L-Q}), $\pi_n(x)=w(x)^nQ_n(t)$.
Thus, it follows that
$$\pi_n(x)\sim w(x)^n(1-t)^{-1/2}$$
as $n\to\infty$. This together with (\ref{L-w}) and (\ref{L-t}) yields (\ref{L-out}).
\end{proof}

\begin{theorem}
Let $\d>0$ be any fixed small number.
For $x$ in a small complex neighborhood of the interval $[-1+\d,1-\d]$, we have
\begin{eqnarray}\label{L-os}
\pi_n(x)\sim{1\over2^n}\left[\cos n\th\left({1+\sin\th\over\sin\th}\right)^{1/2}
+\sin n\th \left({1-\sin\th\over\sin\th}\right)^{1/2}\right]
\end{eqnarray}
as $n\to\infty$, where $\th=\th(x):=\arccos x$ with $0<\re\th<\pi$.
\end{theorem}
\begin{proof}
To put the difference equation (\ref{L-rr}) in the form suggested by Wang and Wong \cite[(2.1)]{WW02}, we let
\begin{eqnarray}\label{L-p}
p_n(x):={2^n\G(n/2+1/4)\G(n/2+3/4)\over[\G(n/2+1/2)]^2}\pi_n(x).
\end{eqnarray}
From (\ref{L-p}), it is easily seen that
\begin{eqnarray}\label{L-p-rr}
{[\G(n/2+1/2)]^2(n/2+1/4)\over[\G(n/2+1)]^2}\cdot2xp_n(x)=p_{n+1}(x)+p_{n-1}(x).
\end{eqnarray}
Motivated by the form of the normal (series) solutions to second-order difference equations (see \cite[(1.5)]{WL92}), we assume
\begin{eqnarray}\label{L-p-asymp1}
p_n(x)\sim n^\a[r(x)]^n\{f(x)\cos[n\vp(x)]+g(x)\sin[n\vp(x)]
\end{eqnarray}
as $n\to\infty$, where $r(x)$ and $\vp(x)$ are real-valued functions,
whereas $f(x)$ and $g(x)$ can be complex-valued.
We now proceed to determine the constant $\a$ and the functions $r(x)$, $\vp(x)$, $f(x)$ and $g(x)$ in (\ref{L-p-asymp1}).
It can be easily shown from (\ref{L-p-asymp1}) that
\begin{eqnarray}\label{L-p-pm-asymp1}
p_{n\pm1}(x)\sim n^\a r^{n\pm1}[(f\cos\vp\pm g\sin\vp)\cos(n\vp)+(g\cos\vp\mp f\sin\vp)\sin(n\vp)].
\end{eqnarray}
Furthermore, by the asymptotic formula for the ratio of Gamma functions \cite[(6.1.47)]{AS70} we have
\begin{eqnarray}\label{L-n}
{[\G(n/2+1/2)]^2(n/2+1/4)\over[\G(n/2+1)]^2}=1+O(n^{-2})
\end{eqnarray}
as $n\to\infty$.
Applying (\ref{L-p-asymp1}), (\ref{L-p-pm-asymp1}) and (\ref{L-n}) to (\ref{L-p-rr}) gives
\begin{eqnarray*}
2x[f\cos(n\vp)+g\sin(n\vp)]&\sim& r[(f\cos\vp+g\sin\vp)\cos(n\vp)+(g\cos\vp-f\sin\vp)\sin(n\vp)]
\\&&+r^{-1}[(f\cos\vp-g\sin\vp)\cos(n\vp)+(g\cos\vp+f\sin\vp)\sin(n\vp)].
\end{eqnarray*}
Comparing the coefficients of $\cos(n\vp)$ and $\sin(n\vp)$ on both sides of the last formula yields
\begin{eqnarray*}
2xf=(r+r^{-1})f\cos\vp+(r-r^{-1})g\sin\vp;\\
2xg=-(r-r^{-1})f\sin\vp+(r+r^{-1})g\cos\vp.
\end{eqnarray*}
Thus, we obtain from the above two equations
$$x=\cosh(\log r)\cos\vp,\qquad\qquad 0=\sinh(\log r)\sin\vp.$$
It can be easily seen that the only solution to these equations is
$\log r=0$ and $\vp=\arccos x$.
Recall that
\begin{eqnarray}\label{L-th}
\th=\th(x):=\arccos x,\qquad\qquad 0<\re\th<\pi.
\end{eqnarray}
Hence, we conclude that
\begin{eqnarray}\label{L-r}
r=1\AND\vp=\th.
\end{eqnarray}
Next we are going to determine the constant $\a$ in (\ref{L-p-asymp1}).
Applying (\ref{L-r}) to (\ref{L-p-asymp1}) gives
\begin{eqnarray}\label{L-p-asymp2}
p_n(x)\sim n^\a[f\cos(n\th)+g\sin(n\th)],
\end{eqnarray}
and
\begin{eqnarray}\label{L-p-pm-asymp2}
p_{n\pm1}(x)\sim n^\a \left(1\pm{\a\over n}\right)[(f\cos\th\pm g\sin\th)\cos(n\th)+(g\cos\th\mp f\sin\th)\sin(n\th)].
\end{eqnarray}
A combination of (\ref{L-p-rr}), (\ref{L-n}), (\ref{L-p-asymp2}) and (\ref{L-p-pm-asymp2}) yields
\begin{eqnarray*}
2x[f\cos(n\th)+g\sin(n\th)]\sim2f\cos\th\cos(n\th)+2g\cos\th\sin(n\th)
+{\a\over n}[2g\sin\th\cos(n\th)-2f\sin\th\sin(n\th)].
\end{eqnarray*}
In view of (\ref{L-th}), we obtain by matching the coefficients in the last formula
$$\a g\sin\th=0,\qquad \a f\sin\th=0.$$
These equations hold for all $x$ in a small complex neighborhood of $[-1+\d,1-\d]$.
Since $f$ and $g$ can not be identically zero, it follows that
\begin{eqnarray}\label{L-a}
\a=0.
\end{eqnarray}
Thus, we have from (\ref{L-p-asymp1}), (\ref{L-r}) and (\ref{L-a})
\begin{eqnarray}\label{L-p-os}
p_n(x)\sim f\cos n\th+g\sin n\th
\end{eqnarray}
as $n\to\infty$.
This formula holds uniformly for $x$ in a small complex neighborhood of $[-1+\d,1-\d]$.
Moreover, it follows from (\ref{L-out}) and (\ref{L-p}) that
\begin{eqnarray}\label{L-p-out}
p_n(x)\sim(x+\sqrt{x^2-1})^n
\left({x+\sqrt{x^2-1}\over2\sqrt{x^2-1}}\right)^{1/2}
\end{eqnarray}
for complex $x$ bounded away from $[-1,1]$.
Our last step is to determine the coefficients $f$ and $g$ in (\ref{L-p-os}) by matching the above two formulas in an overlapping region.
With $\th$ and $x$ given by (\ref{L-th}), it can be shown that
for $\im x>0$, we have $\im\th<0$. Thus, (\ref{L-p-os}) implies
$$p_n(x)\sim \left({f\over2}+{g\over2i}\right)e^{in\th}.$$
Meanwhile, in view of $x=\cos\th$ and $\sqrt{x^2-1}=i\sin\th$ by (\ref{L-th}), we obtain from (\ref{L-p-out}) that
\begin{eqnarray*}
p_n(x)\sim e^{in\th}\left[{e^{i(\th-\pi/2)}\over 2\sin\th}\right]^{1/2}.
\end{eqnarray*}
Coupling the last two formulas gives
$${f\over2}+{g\over2i}={e^{i(\th/2-\pi/4)}\over (2\sin\th)^{1/2}}.$$
Similarly, matching (\ref{L-p-os}) with (\ref{L-p-out}) in the region $\im x<0$ yields
$${f\over2}-{g\over2i}={e^{-i(\th/2-\pi/4)}\over (2\sin\th)^{1/2}}.$$
From the last two equations of $f$ and $g$ we have
$$f=\left({1+\sin\th\over\sin\th}\right)^{1/2},\qquad
g=\left({1-\sin\th\over\sin\th}\right)^{1/2}.$$
This together with (\ref{L-p}) and (\ref{L-p-os}) implies (\ref{L-os}).
\end{proof}

\section{Case 2: the Hermite polynomials}
The Hermite polynomials can be defined as \cite[(1.13.1)]{KS98}
$$H_n(x)=(2x)^n{}_2F_0\left(\begin{array}{c}
    -n/2,-(n-1)/2\\-
  \end{array}
  \bigg|-{1\over x^2}\right).$$
They satisfy the recurrence relation \cite[(1.13.3)]{KS98}
$$2xH_n(x)=H_{n+1}(x)+2nH_{n-1}(x).$$
For convenience, we normalize the Hermite polynomials to be monic, and put
$$\pi_n(x):=2^{-n}H_n(x).$$
The monic Hermite polynomials $\{\pi_n(x)\}_{n=0}^\infty$ satisfy \cite[(1.13.4)]{KS98}
\begin{eqnarray}
&&\pi_{n+1}(x)=x\pi_n(x)-\frac n2\pi_{n-1}(x),\qquad n\ge1,\label{H-rr}\\
&&\pi_0(x)=1,\qquad \pi_1(x)=x.\label{H-ic}
\end{eqnarray}
\begin{theorem}
As $n\to\infty$, we have
\begin{eqnarray}\label{H-out}
\pi_n(\sqrt{2n}y)\sim\left({n\over2e}\right)^{n/2}
\exp\left\{n[y^2-y\sqrt{y^2-1}+\log(y+\sqrt{y^2-1})]\right\}
\left({y+\sqrt{y^2-1}\over2\sqrt{y^2-1}}\right)^{1/2}
\end{eqnarray}
for complex $y$ bounded away from the interval $[-1,1]$.
\end{theorem}
\begin{proof}
Set
\begin{eqnarray}\label{H-wk}
\pi_n(x)=\prod_{k=1}^nw_k(x).
\end{eqnarray}
It follows from (\ref{H-rr}) and (\ref{H-ic}) that $w_1(x)=x$ and
$$w_{k+1}(x)=x-\frac{k}{2w_k(x)}.$$
Let $x=x_n:=\sqrt{2n}y$ with $y\in\C\cut[-1,1]$.
It can be proved by induction that for real $y$ and $y\notin[-1,1]$, we have
\begin{eqnarray*}
{x_n+\sqrt{x_n^2-2k}\over2}\left[1+{1\over2(x_n^2-2k)}
-{5x_n-\sqrt{x_n^2-2k}\over8(x_n^2-2k)^{5/2}}\right]
<w_k(x_n)<{x_n+\sqrt{x_n^2-2k}\over2}\left[1+{1\over2(x_n^2-2k)}\right]
\end{eqnarray*}
for all $k=1,\cdots,n$. From these inequalities, it follows that
\begin{eqnarray}\label{H-wk-asymp}
w_k(x_n)=\frac{x_n+\sqrt{x_n^2-2k}}2\left[1+\frac1{2(x_n^2-2k)}+O(n^{-2})\right]
\end{eqnarray}
as $n\to\infty$, uniformly in $k=1,\cdots,n$.
By using a continuity argument, it can be shown that the validity of this asymptotic formula can be extended to complex $y\in\C\cut[-1,1]$.
Recall that $x_n=\sqrt{2n}y$. By the trapezoidal rule
$$\frac 1n\sum_{k=1}^n f(k/n)=\int_0^1f(t)dt+\frac{f(1)-f(0)}{2n}+O(n^{-2}),$$
we have
\begin{eqnarray}
  \frac 1n\sum_{k=1}^n\log\frac{x_n+\sqrt{x_n^2-2k}}{x_n+\sqrt{x_n^2-2n}}
  &=&
  \frac 1n\sum_{k=1}^n\log(y+\sqrt{y^2-k/n})-\log(y+\sqrt{y^2-1})
  \nonumber\\
  &\sim&\int_0^1\log(y+\sqrt{y^2-t})dt+\frac1{2n}\log\frac{y+\sqrt{y^2-1}}{2y}-\log(y+\sqrt{y^2-1})
  \nonumber\\
  &=&y^2-1/2-y\sqrt{y^2-1}+\frac1{2n}\log\frac{y+\sqrt{y^2-1}}{2y}
\end{eqnarray}
and
\begin{eqnarray}
\sum_{k=1}^n\log\left[1+\frac1{2(x_n^2-2k)}\right]
&\sim&\sum_{k=1}^n\frac1{2(x_n^2-2k)}
={1\over n}\sum_{k=1}^n\frac1{4(y^2-k/n)}\nonumber\\
&\sim&\int_0^1\frac{dt}{4(y^2-t)}
={1\over4}\log{y^2\over y^2-1}
\end{eqnarray}
as $n\to\infty$. Applying the last two formulas and (\ref{H-wk-asymp}) to (\ref{H-wk}) yields
\begin{eqnarray*}
\pi_n(x_n)&\sim&\prod_{k=1}^n\left[\frac{x_n+\sqrt{x_n^2-2n}}2\right]\cdot
\prod_{k=1}^n\left[\frac{x_n+\sqrt{x_n^2-2k}}{x_n+\sqrt{x_n^2-2n}}\right]\cdot
\prod_{k=1}^n\left[1+\frac1{2(x_n^2-2k)}\right]
\\&\sim&\left({n\over2}\right)^{n/2}\left(y+\sqrt{y^2-1}\right)^n
\exp\left[n(y^2-1/2-y\sqrt{y^2-1})\right]\left(\frac{y+\sqrt{y^2-1}}{2y}\right)^{1/2}
\left({y^2\over y^2-1}\right)^{1/4}
\\&\sim&\left({n\over2e}\right)^{n/2}
\exp\left\{n[y^2-y\sqrt{y^2-1}+\log(y+\sqrt{y^2-1})]\right\}
\left({y+\sqrt{y^2-1}\over2\sqrt{y^2-1}}\right)^{1/2},
\end{eqnarray*}
thus proving (\ref{H-out}).
\end{proof}

\begin{theorem}\label{thm-H-os}
Let $\d>0$ be any fixed small number.
For $y$ in a small complex neighborhood of $[-1+\d,1-\d]$, we have
\begin{eqnarray}\label{H-os}
\pi_n(\sqrt{2n}y)\sim\left({n\over2e}\right)^{n/2}{e^{ny^2}\over(1-y^2)^{1/4}}
\left\{\cos\left[n(\th-\sin\th\cos\th)+{\th\over2}\right]+
\sin\left[n(\th-\sin\th\cos\th)+{\th\over2}\right]\right\}
\end{eqnarray}
as $n\to\infty$, where $\th=\th(y):=\arccos y$.
\end{theorem}

To prove the above theorem, we need a lemma analogous to Lemma 1 in \cite{WW02}.
For convenience, we use the notation
\begin{eqnarray}\label{H-y-pm}
y_\pm:=\left({n\over n\pm1}\right)^{1/2}y\sim y\mp{y\over 2n}+{3y\over 8n^2}.
\end{eqnarray}
\begin{lemma}
Let $\vp(y)$ be any analytic function in a small complex neighborhood of $[\d,1-\d]$,
we have
\begin{eqnarray}\label{H-cos-asymp}
\cos[(n\pm1)\vp(y_\pm)]&\sim&\cos(n\vp)\left(\cos\l\mp{\mu\over n}\sin\l\right)\mp\sin(n\vp)\left(\sin\l\pm{\mu\over n}\cos\l\right)
\end{eqnarray}
and
\begin{eqnarray}\label{H-sin-asymp}
\sin[(n\pm1)\vp(y_\pm)]&\sim&\sin(n\vp)\left(\cos\l\mp{\mu\over n}\sin\l\right)\pm\cos(n\vp)\left(\sin\l\pm{\mu\over n}\cos\l\right)
\end{eqnarray}
as $n\to\infty$, where
\begin{eqnarray}\label{H-l}
\l=\l(y):=\vp(y)-{y\vp'(y)\over2},
\end{eqnarray}
and
\begin{eqnarray}\label{H-mu}
\mu=\mu(y):=-{y\vp'(y)\over8}+{y^2\vp''(y)\over8}.
\end{eqnarray}
\end{lemma}

\begin{proof}
From (\ref{H-y-pm}) we have
\begin{eqnarray*}
(n\pm1)\vp(y_\pm)&\sim& n(1\pm{1\over n})\vp(y\mp{y\over2n}+{3y\over8n^2})
\\&\sim& n(1\pm{1\over n})(\vp\mp{y\vp'\over2n}+{3y\vp'\over8n^2}+{y^2\vp''\over8n^2})
\\&\sim& n(\vp\pm{\l\over n}+{\mu\over n^2}),
\end{eqnarray*}
where $\vp$ denotes $\vp(y)$, and $\l$ and $\mu$ are given in (\ref{H-l}) and (\ref{H-mu}).
It then follows that
\begin{eqnarray*}
\cos[(n\pm1)\vp(y_\pm)]&\sim&\cos(n\vp)\cos(\l\pm\mu/n)\mp\sin(n\vp)\sin(\l\pm\mu/n)
\nonumber\\&\sim&\cos(n\vp)\left(\cos\l\mp{\mu\over n}\sin\l\right)\mp\sin(n\vp)\left(\sin\l\pm{\mu\over n}\cos\l\right);
\\
\sin[(n\pm1)\vp(y_\pm)]&\sim&\sin(n\vp)\cos(\l\pm\mu/n)\pm\cos(n\vp)\sin(\l\pm\mu/n)
\nonumber\\&\sim&\sin(n\vp)\left(\cos\l\mp{\mu\over n}\sin\l\right)\pm\cos(n\vp)\left(\sin\l\pm{\mu\over n}\cos\l\right).
\end{eqnarray*}
This proves the lemma.
\end{proof}

\begin{proof}[Proof of Theorem \ref{thm-H-os}]
Define
\begin{eqnarray}\label{H-p}
p_n(x):=[\G(n/2+1/2)]^{-1}\pi_n(x).
\end{eqnarray}
We make a change of variable $x=x_n:=\sqrt{2n}y$.
It is easily seen from (\ref{H-rr}) and (\ref{H-p}) that
\begin{eqnarray}\label{H-p-rr}
{\G(n/2+1/2)\sqrt{2n}\over\G(n/2+1)}\cdot yp_n(\sqrt{2n}y)=p_{n+1}(\sqrt{2n}y)
+p_{n-1}(\sqrt{2n}y).
\end{eqnarray}
As in (\ref{L-p-asymp1}), we now assume
\begin{eqnarray}\label{H-Q}
p_n(\sqrt{2n}y)\sim n^\a [r(y)]^n\{f(y)\cos [n\vp(y)]+g(y)\sin [n\vp(y)]\}
\end{eqnarray}
as $n\to\infty$.
First, we shall determine the constant $\a$ and the functions $r(y)$ and $\vp(y)$ in (\ref{H-Q}).
From (\ref{H-y-pm}) and (\ref{H-Q}), we have
\begin{eqnarray}\label{H-Q-pm}
p_{n\pm1}(\sqrt{2n}y)&=&p_{n\pm1}(\sqrt{2(n\pm1)}y_\pm)
\nonumber\\&\sim& (n\pm1)^\a [r(y_\pm)]^{n\pm1}\{f(y_\pm)\cos[(n\pm1)\vp(y_\pm)]+g(y_\pm)\sin[(n\pm1)\vp(y_\pm)]\}.
\end{eqnarray}
Moreover, it can be shown from (\ref{H-y-pm}) that
\begin{eqnarray}\label{H-r-asymp1}
[r(y_\pm)]^{n\pm1}\sim r^{n\pm1}e^{\mp{yr'/2r}},
\end{eqnarray}
where $r=r(y)$.
Applying (\ref{H-y-pm}), (\ref{H-cos-asymp}), (\ref{H-sin-asymp}) and (\ref{H-r-asymp1}) to (\ref{H-Q-pm}) yields
\begin{eqnarray}\label{H-Q-pm-asymp1}
p_{n\pm1}(\sqrt{2n}y)\sim n^\a r^{n\pm1}e^{\mp{yr'/2r}}
[(f\cos\l\pm g\sin\l)\cos(n\vp)+(g\cos\l\mp f\sin\l)\sin(n\vp)].
\end{eqnarray}
Here $f$ and $g$ stand for $f(y)$ and $g(y)$.
By Stirling's formula \cite[(6.1.37)]{AS70} we have
\begin{eqnarray}\label{H-n}
{\G(n/2+1/2)\sqrt{n/2}\over\G(n/2+1)}\sim1-{1\over4n}.
\end{eqnarray}
A combination of (\ref{H-p-rr}), (\ref{H-Q}), (\ref{H-Q-pm-asymp1}) and (\ref{H-n}) implies
\begin{eqnarray*}
2y[f\cos(n\vp)+g\sin(n\vp)]&\sim&
re^{-{yr'/2r}}[(f\cos\l+g\sin\l)\cos(n\vp)+(g\cos\l-f\sin\l)\sin(n\vp)]
\\&&+
r^{-1}e^{yr'/2r}[(f\cos\l-g\sin\l)\cos(n\vp)+(g\cos\l+f\sin\l)\sin(n\vp)].
\end{eqnarray*}
Comparing the coefficients of $\cos(n\vp)$ and $\sin(n\vp)$ on both sides of the last formula gives
\begin{eqnarray*}
2yf=re^{-{yr'/2r}}(f\cos\l+g\sin\l)+r^{-1}e^{yr'/2r}(f\cos\l-g\sin\l);\\
2yg=re^{-{yr'/2r}}(g\cos\l-f\sin\l)+r^{-1}e^{yr'/2r}(g\cos\l+f\sin\l).
\end{eqnarray*}
Solving the above two equations, we obtain
$$\cosh\left(\log r-{yr'\over2r}\right)\cos\l=y,\qquad
\sinh\left(\log r-{yr'\over2r}\right)\sin\l=0.$$
A solution is
\begin{eqnarray}\label{H-r-eq}
\log r-{yr'\over2r}=0,\qquad \cos\l=y.
\end{eqnarray}
The first equation in (\ref{H-r-eq}) implies
\begin{eqnarray}\label{H-r}
r=e^{cy^2}
\end{eqnarray}
for some constant $c\in\C$.
From (\ref{H-l}) and (\ref{H-r-eq}), we have
$$\vp=\pm(\arccos y-y\sqrt{1-y^2})+c'y^2+2k\pi$$
for some constant $c'\in\C$ and $k\in\N$.
Without loss of generality, we may take $c'=0$ and $k=0$.
Hence,
\begin{eqnarray}\label{H-vp}
\vp=\arccos y-y\sqrt{1-y^2}.
\end{eqnarray}
Next, we are going to determine the functions $f$ and $g$ in (\ref{H-Q}).
From (\ref{H-y-pm}) and (\ref{H-r}) we obtain
\begin{eqnarray}\label{H-r-asymp2}
[r(y_\pm)]^{n\pm1}=[r(y)]^n
\end{eqnarray}
and
\begin{eqnarray}\label{H-fg-asymp}
f(y_\pm)\sim f\mp{yf'\over2n},\qquad g(y_\pm)\sim g\mp{yg'\over2n},
\end{eqnarray}
where $f=f(y)$ and $g=g(y)$.
Applying (\ref{H-cos-asymp}), (\ref{H-sin-asymp}), (\ref{H-r-asymp2}) and (\ref{H-fg-asymp}) to (\ref{H-Q-pm}) yields
\begin{eqnarray}\label{H-Q-pm-asymp2}
{p_{n\pm1}(\sqrt{2n}y)\over n^\a r^n}&\sim&(f\cos\l\pm g\sin\l)\cos(n\vp)+(g\cos\l\mp f\sin\l)\sin(n\vp)
\nonumber\\&&+{\cos(n\vp)\over n}
\left(\pm\a f\cos\l\mp\mu f\sin\l\mp{yf'\over2}\cos\l+\a g\sin\l+\mu g\cos\l-{yg'\over2}\sin\l\right)
\nonumber\\&&+{\sin(n\vp)\over n}\left(\pm\a g\cos\l\mp\mu g\sin\l\mp{yg'\over2}\cos\l-\a f\sin\l-\mu f\cos\l+{yf'\over2}\sin\l\right).
\nonumber\\
\end{eqnarray}
A combination of (\ref{H-p-rr}), (\ref{H-Q}), (\ref{H-n}) and (\ref{H-Q-pm-asymp2}) gives
\begin{eqnarray*}
\left(1-{1\over4n}\right)[2yf\cos(n\vp)+2yg\sin(n\vp)]&\sim& 2f\cos\l\cos(n\vp)+2g\cos\l\sin(n\vp)
\\&&+{\cos(n\vp)\over n}(2\a g\sin\l+2\mu g\cos\l-yg'\sin\l)
\\&&
+{\sin(n\vp)\over n}(-2\a f\sin\l-2\mu f\cos\l+yf'\sin\l).
\end{eqnarray*}
In view of the second equation in (\ref{H-r-eq}), we obtain by matching the coefficients in the last formula
\begin{eqnarray}
{f\over2}\cos\l+2\a g\sin\l+2\mu g\cos\l-yg'\sin\l=0;
\label{H-eq1}\\
{g\over2}\cos\l-2\a f\sin\l-2\mu f\cos\l+yf'\sin\l=0.
\label{H-eq2}
\end{eqnarray}
Note from (\ref{H-l}), (\ref{H-mu}) and (\ref{H-vp}) that
$$\l=\arccos y,\qquad \mu={y\over4\sqrt{1-y^2}}.$$
Hence, equations (\ref{H-eq1}) and (\ref{H-eq2}) can be written as
\begin{eqnarray}
f+{4\a g\sqrt{1-y^2}\over y}+{yg\over\sqrt{1-y^2}}-2g'\sqrt{1-y^2}=0;
\label{H-fg-eq1}\\
g-{4\a f\sqrt{1-y^2}\over y}-{yf\over\sqrt{1-y^2}}+2f'\sqrt{1-y^2}=0.
\label{H-fg-eq2}
\end{eqnarray}
Set
\begin{eqnarray}\label{H-fg}
u:=y^{-2\a}(1-y^2)^{1/4}f;\qquad v:=y^{-2\a}(1-y^2)^{1/4}g.
\end{eqnarray}
We then have from (\ref{H-fg-eq1}), (\ref{H-fg-eq2}) and (\ref{H-fg})
\begin{eqnarray}\label{H-uv-eq}
u'=-{v\over2\sqrt{1-y^2}};\qquad v'={u\over2\sqrt{1-y^2}}.
\end{eqnarray}
Define
\begin{eqnarray}\label{H-th}
\th=\th(y):=\arccos y.
\end{eqnarray}
The solution of the system (\ref{H-uv-eq}) is given by
\begin{eqnarray}
u=C_1\cos{\th\over2}+C_2\sin{\th\over2};\qquad
v=-C_1\sin{\th\over2}+C_2\cos{\th\over2},
\end{eqnarray}
where $C_1\in\C$ and $C_2\in\C$ are two arbitrary constants.
Consequently, we obtain from (\ref{H-fg}) that
\begin{eqnarray}
f&=&{y^{2\a}\over(1-y^2)^{1/4}}\left(C_1\cos{\th\over2}+C_2\sin{\th\over2}\right);
\label{H-f}\\
g&=&{y^{2\a}\over(1-y^2)^{1/4}}\left(-C_1\sin{\th\over2}+C_2\cos{\th\over2}\right).
\label{H-g}
\end{eqnarray}
Applying (\ref{H-r}), (\ref{H-f}) and (\ref{H-g}) to (\ref{H-Q}) yields
\begin{eqnarray}\label{H-p-os}
p_n(\sqrt{2n}y)\sim n^\a e^{ncy^2}y^{2\a}(1-y^2)^{-1/4}[C_1\cos(n\vp+\th/2)+C_2\sin(n\vp+\th/2)].
\end{eqnarray}
This formula holds uniformly for $y$ in a small complex neighborhood of $[-1+\d,1-\d]$.
Moreover, it follows from (\ref{H-out}) and (\ref{H-p}) that
\begin{eqnarray}\label{H-p-out}
p_n(\sqrt{2n}y)\sim{1\over\sqrt{2\pi}}\exp\left\{n[y^2-y\sqrt{y^2-1}+\log(y+\sqrt{y^2-1})]\right\}
\left({y+\sqrt{y^2-1}\over2\sqrt{y^2-1}}\right)^{1/2}
\end{eqnarray}
for complex $y$ bounded away from $[-1,1]$.
Finally, we match the above two formulas in an overlapping region to determine the constants $\a$, $c$, $C_1$ and $C_2$ in (\ref{H-p-os}).
For $\im y>0$, it follows from (\ref{H-th}) that $\im\th<0$;
see a similar statement following (\ref{L-p-out}).
Furthermore, it can be shown from (\ref{H-vp}) that if $\im y>0$,
then we also have $\im\vp<0$.
[To do this, one first notes that $\vp'(y)$ is negative for $y\in[-1+\d,1-\d]$.
Then, by the continuity of $\vp'$,
one concludes that $\re\vp'(y)<0$ for $y$ in a neighborhood of $[-1+\d,1-\d]$ in the complex plane.
Finally, the mean value theorem ensures that there exists a real number $\xi\in(0,\im y)$ such that $\vp(y)=\vp(\re y)+i(\im y)\vp'(\re y+i\xi)$,
from which one obtains $\im\vp(y)<0$.]
Thus, (\ref{H-p-os}) implies
$$p_n(\sqrt{2n}y)\sim n^\a e^{ncy^2}y^{2\a}(1-y^2)^{-1/4}\left({C_1\over2}+{C_2\over2i}\right)e^{in\vp+i\th/2}.$$
Meanwhile, we have from (\ref{H-th}) and (\ref{H-p-out})
$$p_n(\sqrt{2n}y)\sim{1\over\sqrt{2\pi}}\exp\left\{n[y^2-iy\sqrt{1-y^2}+i\arccos y]\right\}
\left[{e^{i(\th-\pi/2)}\over2\sqrt{1-y^2}}\right]^{1/2}.$$
Thus, we obtain from (\ref{H-vp}) and the above two formulas that $\a=0$, $c=1$ and
$${C_1\over2}+{C_2\over2i}={e^{-i\pi/4}\over2\sqrt{\pi}}.$$
Similarly, matching (\ref{H-p-os}) with (\ref{H-p-out}) in the region $\im y<0$ yields again $\a=0$, $c=1$ and the equation
$${C_1\over2}-{C_2\over2i}={e^{i\pi/4}\over2\sqrt{\pi}}.$$
Coupling the last two equations gives
$$C_1=C_2={1\over\sqrt{2\pi}}.$$
Therefore, we conclude that
$$\a=0,\qquad c=1,\qquad C_1=C_2={1\over\sqrt{2\pi}}.$$
This together with (\ref{H-p}), (\ref{H-vp}) and (\ref{H-p-os}) yields (\ref{H-os}).
\end{proof}

\section{Case 3: an open problem}
Recently, M. E. H. Ismail proposed the problem of finding asymptotic formulas for the orthogonal polynomials determined by
\begin{eqnarray}
&&\pi_{n+1}(x)=(x-n^2)\pi_n(x)-\frac 14\pi_{n-1}(x),\qquad n\ge1,\label{I-rr}\\
&&\pi_0(x)=1,\qquad \pi_1(x)=x;\label{I-ic}
\end{eqnarray}
see \cite[\S 6]{IK} and \cite[p. 370]{GM02}.
We first present a result for $x$ not in the interval of oscillation.
\begin{theorem}
As $n\to\infty$, we have
\begin{eqnarray}\label{I-out}
\pi_n(n^2y)\sim\left({n\over e}\right)^{2n}
\exp\left\{n[(\sqrt y+1)\log(\sqrt y+1)-(\sqrt y-1)\log(\sqrt y-1)]\right\}
\left({y\over y-1}\right)^{1/2}
\end{eqnarray}
for complex $y$ bounded away from $[0,1]$.
\end{theorem}
\begin{proof}
Set
\begin{eqnarray}\label{I-wk}
\pi_n(x)=\prod_{k=1}^nw_k(x).
\end{eqnarray}
It follows from (\ref{I-rr}) and (\ref{I-ic}) that $w_1(x)=x$ and
$$w_{k+1}(x)=x-k^2-\frac{1}{4w_k(x)}.$$
Let $x=x_n:=n^2y$ with $y\in\C\cut[0,1]$.
As with the case of Hermite polynomials,
it can be shown that for real $x$ and $x\notin[0,n^2]$, we have
$$x-(k-1)^2-1<w_k(x)<x-(k-1)^2+1$$
for all $k=1,\cdots,n$. Thus,
$$1+{2k\over x-k^2}-{2\over x-k^2}<{w_k(x)\over x-k^2}<1+{2k\over x-k^2}.$$
Consequently,
\begin{eqnarray}\label{I-wk-asymp}
w_k(n^2y)=n^2\left(y-{k^2\over n^2}\right)\left[1+{2k\over n^2y-k^2}+O(n^{-2})\right]
\end{eqnarray}
as $n\to\infty$, uniformly in $k=1,\cdots,n$.
By using a continuity argument, it can be shown that the validity of this asymptotic formula can be extended to complex $y\in\C\cut[0,1]$.
In view of the trapezoidal rule
$$\frac 1n\sum_{k=1}^n f(k/n)\sim \int_0^1f(t)dt+\frac{f(1)-f(0)}{2n},$$
we have
\begin{eqnarray*}
\sum_{k=1}^n\log\left(y-{k^2\over n^2}\right)&\sim& n\int_0^1\log(y-t^2)dt+{1\over2}\log{y-1\over y}
\\&=&n[(\sqrt y+1)\log(\sqrt y+1)-(\sqrt y-1)\log(\sqrt y-1)-2]+{1\over2}\log{y-1\over y}
\end{eqnarray*}
and
\begin{eqnarray*}
\sum_{k=1}^n\log\left(1+{2k\over n^2y-k^2}\right)
\sim\sum_{k=1}^n{2k\over n^2y-k^2}
\sim\int_0^1{2t\over y-t^2}dt=\log{y\over y-1}
\end{eqnarray*}
as $n\to\infty$.
Applying the last two formulas and (\ref{I-wk-asymp}) to (\ref{I-wk}) gives (\ref{I-out}).
\end{proof}
Next we give a result for $x$ inside the interval of oscillation.
\begin{theorem}\label{thm-I-os}
Let $\d>0$ be any fixed small number.
For $y$ in a small neighborhood of $[\d,1-\d]$ in the complex plane, we have
\begin{eqnarray}\label{I-os}
\pi_n(n^2y)\sim(-1)^{n-1}2\sin(n\pi\sqrt y)\left({n\over e}\right)^{2n}
\left({1+\sqrt y\over1-\sqrt y}\right)^{n\sqrt y}y^{1/2}(1-y)^{n-1/2}
\end{eqnarray}
as $n\to\infty$.
\end{theorem}

To prove the above theorem, we will need a lemma analogous to Lemma 1 in \cite{WW02}.
As in (\ref{H-y-pm}), for convenience we set
\begin{eqnarray}\label{I-y-pm}
y_\pm:=\left({n\over n\pm1}\right)^2y\sim y\mp{2y\over n}+{3y\over n^2}.
\end{eqnarray}
\begin{lemma}\label{I-lem}
Let $\vp(y)$ be any analytic function in a small neighborhood of $[\d,1-\d]$ in the complex plane,
we have
\begin{eqnarray}\label{I-cos-asymp}
\cos[(n\pm1)\vp(y_\pm)]&\sim&\cos(n\vp)\left(\cos\l\mp{\mu\over n}\sin\l\right)\mp\sin(n\vp)\left(\sin\l\pm{\mu\over n}\cos\l\right)
\end{eqnarray}
and
\begin{eqnarray}\label{I-sin-asymp}
\sin[(n\pm1)\vp(y_\pm)]&\sim&\sin(n\vp)\left(\cos\l\mp{\mu\over n}\sin\l\right)\pm\cos(n\vp)\left(\sin\l\pm{\mu\over n}\cos\l\right)
\end{eqnarray}
as $n\to\infty$,
where
\begin{eqnarray}\label{I-l}
\l=\l(y):=\vp(y)-2y\vp'(y),
\end{eqnarray}
and
\begin{eqnarray}\label{I-mu}
\mu=\mu(y):=y\vp'(y)+2y^2\vp''(y).
\end{eqnarray}
\end{lemma}
\begin{proof}
From (\ref{I-y-pm}) we have
\begin{eqnarray*}
(n\pm1)\vp(y_\pm)&\sim& n(1\pm{1\over n})\vp(y\mp{2y\over n}+{3y\over n^2})
\\&\sim& n(1\pm{1\over n})(\vp\mp{2y\vp'\over n}+{3y\vp'\over n^2}+{2y^2\vp''\over n^2})
\\&\sim& n(\vp\pm{\l\over n}+{\mu\over n^2}),
\end{eqnarray*}
where $\l$ and $\mu$ are given in (\ref{I-l}) and (\ref{I-mu}).
It then follows that
\begin{eqnarray*}
\cos[(n\pm1)\vp(y_\pm)]&\sim&\cos(n\vp)\cos(\l\pm\mu/n)\mp\sin(n\vp)\sin(\l\pm\mu/n)
\nonumber\\&\sim&\cos(n\vp)\left(\cos\l\mp{\mu\over n}\sin\l\right)\mp\sin(n\vp)\left(\sin\l\pm{\mu\over n}\cos\l\right);
\\
\sin[(n\pm1)\vp(y_\pm)]&\sim&\sin(n\vp)\cos(\l\pm\mu/n)\pm\cos(n\vp)\sin(\l\pm\mu/n)
\nonumber\\&\sim&\sin(n\vp)\left(\cos\l\mp{\mu\over n}\sin\l\right)\pm\cos(n\vp)\left(\sin\l\pm{\mu\over n}\cos\l\right).
\end{eqnarray*}
This proves the lemma.
\end{proof}

\begin{proof}[Proof of Theorem \ref{thm-I-os}]
Define
\begin{eqnarray}\label{I-p}
p_n(x):={(-1)^n\over\G(n)^2}\pi_n(x).
\end{eqnarray}
We make a change of variable $x=x_n:=n^2y$.
It is readily seen from (\ref{I-rr}) and (\ref{I-p}) that
\begin{eqnarray}\label{I-p-rr}
(1-y)p_n(n^2y)=p_{n+1}(n^2y)
+{1\over4n^2(n-1)^2}p_{n-1}(n^2y).
\end{eqnarray}
As in (\ref{H-Q}), we first assume
\begin{eqnarray}\label{I-Q}
p_n(n^2y)\sim n^\a[r(y)]^n\{f(y)\cos[n\vp(y)]+g(y)\sin[n\vp(y)]\}
\end{eqnarray}
as $n\to\infty$, and then determine the constant $\a$ and the functions $r(y)$, $f(y)$, $g(y)$ and $\vp(y)$ in the formula.
From (\ref{I-y-pm}) and (\ref{I-Q}) we have
\begin{eqnarray}\label{I-Q-pm}
p_{n\pm1}(n^2y)&=&p_{n\pm1}((n\pm1)^2y_\pm)
\nonumber\\&\sim& (n\pm1)^\a [r(y_\pm)]^{n\pm1}\{f(y_\pm)\cos[(n\pm1)\vp(y_\pm)]+g(y_\pm)\sin[(n\pm1)\vp(y_\pm)]\}.
\end{eqnarray}
Moreover, it can be shown from (\ref{I-y-pm}) that as $n\to\infty$, we also have
\begin{eqnarray}\label{I-r-asymp}
[r(y_\pm)]^{n\pm1}\sim r^{n\pm1}e^{\mp2y{r'/ r}},
\end{eqnarray}
where $r$ stands for $r(y)$.
Applying (\ref{I-y-pm}), (\ref{I-cos-asymp}), (\ref{I-sin-asymp}) and (\ref{I-r-asymp}) to (\ref{I-Q-pm}) yields
\begin{eqnarray}\label{I-Q-pm-asymp}
p_{n\pm1}(n^2y)\sim n^\a r^{n\pm1}e^{\mp{2yr'/ r}}
[(f\cos\l\pm g\sin\l)\cos(n\vp)+(g\cos\l\mp f\sin\l)\sin(n\vp)].
\end{eqnarray}
A combination of (\ref{I-p-rr}), (\ref{I-Q}) and (\ref{I-Q-pm-asymp}) gives
\begin{eqnarray*}
(1-y)[f\cos(n\vp)+g\sin(n\vp)]\sim
re^{-{2yr'/ r}}[(f\cos\l+g\sin\l)\cos(n\vp)+(g\cos\l-f\sin\l)\sin(n\vp)].
\end{eqnarray*}
By comparing the coefficients of $\cos(n\vp)$ and $\sin(n\vp)$ on both sides of the last formula, we obtain
\begin{eqnarray*}
(1-y)f=re^{-{2yr'/ r}}(f\cos\l+g\sin\l);\\
(1-y)g=re^{-{2yr'/ r}}(g\cos\l-f\sin\l).
\end{eqnarray*}
Thus, we have from the above equations
$$(1-y)=re^{-{2yr'/ r}}\cos\l,\qquad 0=re^{-{2yr'/ r}}\sin\l.$$
The only solution is $\l=0$, and
\begin{eqnarray}\label{I-r-eq}
re^{-{2yr'/ r}}=1-y.
\end{eqnarray}
With $\l=0$, we obtain from (\ref{I-l})
\begin{eqnarray}\label{I-vp}
\vp=c\sqrt y
\end{eqnarray}
for some constant $c\in\C$.
Let $R(y):=\log r(y)$. From (\ref{I-r-eq}),
it is easily seen that $R(y)$ satisfies a first-order linear inhomogeneous equation,
whose solution is given by
$$R(y)=-{1\over2}y^{1/2}\left[\int^y s^{-3/2}\log(1-s)ds\right].$$
Upon integration by parts, followed by a change of variable $u=s^{1/2}$,
one obtains
$$R(y)=\log(1-y)+2y^{1/2}\arctanh\sqrt y+c'\sqrt y$$
for some constant $c'\in\C$.
Taking exponential on both sides gives
\begin{eqnarray*}
r(y)=(1-y)\left({1+\sqrt y\over1-\sqrt y}\right)^{\sqrt y}e^{c'\sqrt y}.
\end{eqnarray*}
Without loss of generality, we may assume $c'=0$. Hence,
\begin{eqnarray}\label{I-r}
r(y)=(1-y)\left({1+\sqrt y\over1-\sqrt y}\right)^{\sqrt y}.
\end{eqnarray}
Next we determine the functions $f$ and $g$ in (\ref{I-Q}).
From (\ref{I-y-pm}), (\ref{I-r-eq}) and (\ref{I-r}) we have
\begin{eqnarray}
[r(y_\pm)]^{n\pm1}\sim(1-y)^{\pm1}[r(y)]^n\left[1+{y\over n(1-y)}\right].
\end{eqnarray}
Furthermore, it is easily seen from (\ref{I-y-pm}) and (\ref{I-vp}) that
$(n\pm1)\vp(y_\pm)\sim n\vp(y)$
and
$$f(y_\pm)\sim f(y)\mp{2yf'(y)\over n},\qquad g(y_\pm)\sim g(y)\mp{2yg'(y)\over n}.$$
Applying the above formulas for functions $r$, $\vp$, $f$ and $g$ to (\ref{I-Q-pm}) yields
$$p_{n\pm1}(n^2y)\sim n^\a r^n(1-y)^{\pm1}\left(1\pm{\a\over n}\right)\left[1+{y\over n(1-y)}\right]\left[\left(f\mp{2yf'\over n}\right)\cos(n\vp)
+\left(g\mp{2yg'\over n}\right)\sin(n\vp)\right].$$
[One can also obtain this result from Lemma \ref{I-lem}, since $\l=\mu=0$ by (\ref{I-vp}).]
This together with (\ref{I-p-rr}) and (\ref{I-Q}) implies
\begin{eqnarray*}
f\cos(n\vp)+g\sin(n\vp)\sim\left[f+{1\over n}\left({yf\over1-y}+\a f-2yf'\right)\right]\cos(n\vp)
+\left[g+{1\over n}\left({yg\over1-y}+\a g-2yg'\right)\right]\sin(n\vp).
\end{eqnarray*}
Comparing the coefficients on both sides of the last formula gives
$${yf\over1-y}+\a f-2yf'=0,\qquad {yg\over1-y}+\a g-2yg'=0.$$
Hence,
\begin{eqnarray}\label{I-fg}
f=C_1y^{\a/2}(1-y)^{-1/2},\qquad g=C_2y^{\a/2}(1-y)^{-1/2},
\end{eqnarray}
where $C_1\in\C$ and $C_2\in\C$ are two arbitrary constants.
Applying (\ref{I-vp}), (\ref{I-r}) and (\ref{I-fg}) to (\ref{I-Q}) yields
\begin{eqnarray}\label{I-p-os}
p_n(n^2y)\sim n^\a y^{\a/2}(1-y)^{n-1/2}\left({1+\sqrt y\over1-\sqrt y}\right)^{n\sqrt y}[C_1\cos(nc\sqrt y)+C_2\sin(nc\sqrt y)].
\end{eqnarray}
This formula holds uniformly for $y$ in a small neighborhood of $[\d,1-\d]$ in the complex plane.
Moreover, it follows from (\ref{I-out}) and (\ref{I-p}) that
\begin{eqnarray}\label{I-p-out}
p_n(n^2y)\sim{(-1)^nn\over2\pi}
\exp\left\{n[(\sqrt y+1)\log(\sqrt y+1)-(\sqrt y-1)\log(\sqrt y-1)]\right\}
\left({y\over y-1}\right)^{1/2}
\end{eqnarray}
for complex $y$ bounded away from $[0,1]$.
At the final stage, we match the last two formulas in an overlapping region
to determine the constants $\a$, $c$, $C_1$ and $C_2$ in (\ref{I-p-os}).
In view of the equalities $\exp(\pm inc\sqrt y)=\cos(nc\sqrt y)\pm i\sin(nc\sqrt y)$
and
$$(1-y)^n\left({1+\sqrt y\over1-\sqrt y}\right)^{n\sqrt y}=
\exp\left\{n[(\sqrt y+1)\log(\sqrt y+1)-(\sqrt y-1)\log(1-\sqrt y)]\right\},$$
formula (\ref{I-p-os}) can be written as
\begin{eqnarray}\label{I-p-os2}
p_n(n^2y)&\sim& n^\a y^{\a/2}(1-y)^{-1/2}\exp\left\{n[(\sqrt y+1)\log(\sqrt y+1)-(\sqrt y-1)\log(1-\sqrt y)]\right\}
\nonumber\\&&\times\left[\left({C_1\over2}-{C_2\over2i}\right)e^{-inc\sqrt y}+
\left({C_1\over2}+{C_2\over2i}\right)e^{inc\sqrt y}\right].
\end{eqnarray}
Meanwhile, it follows from (\ref{I-p-out}) that for $\im y>0$, we have
$$p_n(n^2y)\sim{n\over2\pi}
\exp\left\{n[(\sqrt y+1)\log(\sqrt y+1)-(\sqrt y-1)\log(1-\sqrt y)]-in\pi\sqrt y-i\pi/2\right\}
\left({y\over 1-y}\right)^{1/2}.$$
A comparison of the above two asymptotic formulas shows that $\a=1$ and $c=\pi$ or $c=-\pi$.
Without loss of generality, we take $c=\pi$.
Note that the function $\exp(inc\sqrt y)=\exp(in\pi\sqrt y)$ is exponentially small, and hence negligible in the region $\im y>0$.
By matching the last two formulas one more time, and ignoring the exponentially small term, we have
$${C_1\over2}-{C_2\over2i}={e^{-i\pi/2}\over2\pi}.$$
With $\a=1$ and $c=\pi$, we match (\ref{I-p-out}) with (\ref{I-p-os2}) in the region $\im y<0$ to obtain the other equation
$${C_1\over2}+{C_2\over2i}={e^{i\pi/2}\over2\pi}.$$
Upon solving the last two equations, we obtain $C_1=0$ and $C_2=-1/\pi$.
Therefore, we conclude that
$$\a=1,\qquad c=\pi,\qquad C_1=0,\qquad C_2=-1/\pi.$$
Combining this with (\ref{I-p}) and (\ref{I-p-os}) gives (\ref{I-os}).
\end{proof}

\end{document}